\documentclass[letterpaper, 11pt]{amsart}
\usepackage{mathtools}
\usepackage{amsmath}
\usepackage[T1]{fontenc}
\usepackage{amssymb}
\usepackage{yhmath}
\usepackage{comment}
\usepackage{graphicx} 
\usepackage{ mathrsfs }
\usepackage{bbm} 
\usepackage{xcolor}
\usepackage{tikz-cd}
\usepackage{tikz}
\usetikzlibrary{patterns}
\usepackage{hyperref}

\DeclareMathAlphabet{\mathpzc}{OT1}{pzc}{m}{it}

\usepackage{thmtools}
\usepackage{thm-restate}

\usepackage{caption}

\newtheorem{theorem}{Theorem}[section]

\newtheorem*{claim*}{Claim}

\newtheorem{lemma}[theorem]{Lemma}

\newtheorem{proposition}[theorem]{Proposition}

\theoremstyle{definition}
\newtheorem{definition}[theorem]{Definition}

\theoremstyle{remark}
\newtheorem{remark}[theorem]{Remark}

\numberwithin{equation}{section}

\DeclareMathOperator{\Isom}{\operatorname{Isom}}

\DeclareMathOperator{\dist}{d}

\DeclareMathOperator{\Fc}{\mathcal{F}}

\DeclareMathOperator{\Tc}{\mathcal{T}}
\DeclareMathOperator{\Uc}{\mathcal{U}}

\DeclareMathOperator{\Nb}{\mathbb{N}}

\DeclareMathOperator{\Rb}{\mathbb{R}}

\DeclareMathOperator{\Zb}{\mathbb{Z}}

\newcommand{\Ga}{\Gamma}

\newcommand{\La}{\Lambda}
\newcommand{\la}{\lambda}

\newcommand{\ba}{\backslash}

\newcommand{\Mod}{\operatorname{Mod}}

\newcommand{\PMF}{\mathcal{PMF}}

\newcommand{\MF}{\mathcal{MF}}
\newcommand{\ML}{\mathcal{ML}}

\newcommand{\CAT}{\operatorname{CAT}}

\newcommand{\inte}{\operatorname{int}}

\newcommand{\len}{\operatorname{len}}

\begin{document}

\title[Non-arithmeticity of length spectra of subgroups of MCG]{Non-arithmeticity of length spectra of subgroups of mapping class groups}

\author{Inhyeok Choi}
\address{School of Mathematics, KIAS, Hoegi-ro 85, Dongdaemun-gu, Seoul 02455, South Korea 
\newline
\indent Simons Laufer Mathematical Sciences Institute, Berkeley, CA 94720}
\email{inhyeokchoi48@gmail.com}

\author{Dongryul M. Kim}
\address{Simons Laufer Mathematical Sciences Institute, Berkeley, CA 94720}
\email{dongryul.kim97@gmail.com}

\date{\today}

\begin{abstract}
In this paper, we prove that every non-elementary subgroup of the mapping class group of a surface has non-arithmetic Teichm\"uller length spectrum. Namely, Teichm\"uller translation lengths of its pseudo-Anosov elements generate a dense additive subgroup of $\mathbb{R}$. We prove this by introducing the notion of cross-ratios on $\mathcal{MF}$ and $\mathcal{PMF}$, and studying its geometric and dynamical properties, despite the lack of negatively curved features of the Teichm\"uller space nor the conformal geometry on $\mathcal{PMF}$.
\end{abstract}

\maketitle


\section{Introduction}

Let $S$ be a connected orientable surface of finite type with negative Euler characteristic. Let
$$
\Mod(S) := \operatorname{Homeo}^+(S) / \sim
$$
be the mapping class group of $S$, the group of isotopy classes of orientation-preserving homeomorphisms on $S$.

Teichm\"uller space $\Tc = \Tc(S)$ is the space of all marked Riemann
surface structures on $S$, equipped with a natural metric $\dist_{\Tc}$, called the Teichm\"uller metric. The natural $\Mod(S)$-action on $(\Tc, \dist_{\Tc})$ given by change of markings is properly discontinuous and by isometries. Hence, for each $g \in \Mod(S)$, its \emph{Teichm\"uller translation length}
$$
\len(g) := \lim_{n \to + \infty} \frac{ \dist_{\Tc}(o, g^n o)}{n}, \quad o \in \Tc
$$
is well-defined, independent of the choice of $o \in \Tc$.

For a subgroup $\Ga < \Mod(S)$, we call
$$
\len(\Ga) := \{ \len(g) : g \in \Ga \text{ pseudo-Anosov}\}
$$
the \emph{Teichm\"uller length spectrum} of $\Ga$. The quotient space $\Mod(S) \ba \Tc$ is the moduli space of Riemann surface structures on $S$, and the Teichm\"uller length spectrum $\len(\Mod(S))$ is the same as the set of lengths of all closed geodesics in the moduli space $\Mod(S) \ba \Tc$, with respect to the induced metric on it. Similarly, the Teichm\"uller length spectrum $\len(\Ga)$ of $\Ga < \Mod(S)$ is the same as the set of lengths of all closed geodesics in the associated cover $\Ga \ba \Tc$ of the moduli space.

We say that $\Ga < \Mod(S)$ has \emph{non-arithmetic} Teichm\"uller length spectrum if the additive subgroup 
$$
\langle \len(\Ga) \rangle < \Rb
$$
generated by its Teichm\"uller length spectrum is dense. We also call $\len(\Ga)$ non-arithmetic in this case.

A subgroup $\Ga < \Mod(S)$ is called \emph{non-elementary} if $\Ga$ contains two independent pseudo-Anosov mapping classes.
Our main theorem is the non-arithmeticity of Teichm\"uller length spectrum of a non-elementary subgroup.

\begin{theorem}[Non-arithmeticity] \label{thm:main}
    Let $\Ga < \Mod(S)$ be a non-elementary subgroup. Then its Teichm\"uller length spectrum $\len(\Ga)$ is non-arithmetic.
\end{theorem}

Note that the non-elementary hypothesis is necessary; if $\Ga$ is a cyclic subgroup generated by a single pseudo-Anosov mapping class, then $\langle\len(\Ga)\rangle \subset \Rb$ is a scaled copy of $\Zb$.

In general, the length spectrum can be similarly discussed for a general isometric action on a metric space. Its non-arithmeticity is a fundamentally important property from many geometric and dynamical aspects.
Indeed, non-arithmeticity is a necessary ingredient for mixing, equidistribution, counting results for geodesic flows, dynamics of horospherical foliations, and measure classification of horospherical-invariant measures. In general negatively curved settings, those geometry and dynamics were studied in (\cite{Dalbo2000topologie}, \cite{Roblin2003ergodicite}, etc). In the settings of subgroups of mapping class groups and Teichm\"uller spaces, study of geodesic flows was carried in \cite{gekhtman2023dynamics}, and measure classifications within $\ML$ was studied in our recent work \cite{CK_ML}. All of those results assume non-arithmeticity of length spectra. 

In the following specific cases of a non-elementary discrete subgroup $\Ga < \Isom(X)$ of isometries on a $\CAT(-1)$ space $X$, non-arithmeticity of its length spectrum is known:
\begin{enumerate}
\item $\Gamma \ba X$ is a negatively curved surface, due to Dal'bo \cite{dalbo1999remarques},
\item $X$ is a rank-one symmetric space, due to Kim\footnote{Inkang Kim, not the second-named author.} \cite{kim2006length}.
\item the limit set of $\Ga$ possesses a non-singleton connected component, due to Bourdon \cite{bourdon1995structure}.
\item $\Gamma$ contains a parabolic isometry, due to Dal'bo--Peign\'e \cite{dalbo1998some}.
\end{enumerate}
 A natural generalization to Zariski dense discrete subgroups of higher-rank Lie groups is established by Benoist \cite{Benoist2000proprietes}. 

 \begin{remark}
 We also note that, to the best of the authors' knowledge, it is an open problem whether the length spectrum of a non-elementary discrete subgroup $\Ga < \Isom(X)$ is always non-arithmetic when $X$ is a simply connected Riemannian manifold with curvature at most $-1$.
 \end{remark}

 In contrast, Teichm\"uller space is neither Gromov hyperbolic \cite{MasurWolf_Teichnothyp}  nor homogeneous except for sporadic surfaces. Besides the full mapping class group $\Mod(S)$ which contains a Veech group acting on a Teichm\"uller disk as a non-elementary Fuchsian group, Theorem \ref{thm:main} is the first result for non-arithmeticity of Teichm\"uller length spectra of  subgroups of $\Mod(S)$.

 \subsection{On the proof}
 To prove Theorem \ref{thm:main}, we introduce the notion of cross-ratios on the space $\mathcal{MF}$ of measured foliations and the space $\mathcal{PMF}$ of projective measured foliations on $S$, and studying its geometric and dynamical properties. See Definition \ref{def:CR} for precise definitions of cross-ratios.

 Showing the non-arithmeticity using cross-ratios was already done by Kim\footnote{Inkang Kim, not the second-named author.} \cite{kim2006length} for rank-one symmetric spaces. Indeed for instance in real hyperbolic case, Kim proved that if a discrete subgroup of isometries does not have non-arithmetic length spectrum, then its limit set is contained in a discrete union of proper spheres in the boundary. This can also be seen as a consequence of negatively curved geometry and the conformality of boundary action. Kim then deduced that the discrete subgroup cannot be non-elementary, based on a structure theory of rank-one Lie groups.

 In contrast, in the setting of Teichm\"uller geometry, we do not have a negatively curved geometry, homogeneity of the space, and conformality of the boundary action. To overcome this difficulty, after we introduce the notion of cross-ratios on $\MF$ and $\PMF$, we also observe useful relationships between cross-ratios and  geometric and dynamical properties of limit sets. Moreover, we further investigate the piecewise-linear structure of $\MF$ and localize dynamics by using the theory of train tracks, and then complete the proof of the non-arithmeticity by proceeding with certain inductive arguments.
 
 \begin{remark}
 We remark that an approach to proving the non-arithmeticity of Teichm\"uller length spectra (Theorem \ref{thm:main}) via train tracks and reduction to linear algebra has been attempted with a mistake in an earlier version \cite{gekhtman2023dynamics_earlier} of the work of Gekhtman--Ma \cite{gekhtman2023dynamics} as well as in Gekhtman's UChicago Ph.D. thesis \cite{Gekhtman_thesis}, which was purported to work even for non-elementary semigroups of mapping class groups. In this current paper, we take a different approach, as described above.
 \end{remark}

\subsection*{Acknowledgements}
The authors thank Mladen Bestvina, Dick Canary, Ursula Hamenst\"adt, Yair Minsky, Hee Oh, and Saul Schleimer for helpful conversations. We also appreciate Ilya Gekhtman and Biao Ma for valuable communication regarding this problem and comments on our earlier draft.

This material is based upon work supported by the National Science Foundation under Grant No. DMS-2424139, while the authors were in residence at the Simons Laufer Mathematical Sciences Institute in Berkeley, California, during the Spring 2026 semester.

Choi was supported by the Mid-Career Researcher Program (RS-2023-00278510) through the National Research Foundation funded by the government of Korea, and by the KIAS individual grant (MG091901) at KIAS.

 \section{Proof of the non-arithmeticity}

\subsection{Measured foliations and pseudo-Anosov mapping classes}

Before proving Theorem \ref{thm:main}, let us briefly review some basic facts on measured foliations and pseudo-Anosov mapping classes. For more comprehensive overview, we refer to (\cite{1979travaux}, \cite{farb2012primer}). 

We denote by $\MF = \MF(S)$ the space of equivalence classes of (singular) measured foliations on $S$, where the equivalence is generated by isotopy and Whitehead moves (see \cite[Expos{\'e} 5]{1979travaux} for details). Each element of $\MF$ can be represented by a pair $(\Fc, \mu)$ of (singular) foliation $\Fc$ on $S$ and a transverse measure $\mu$. 

The mapping class group $\Mod(S)$ acts naturally on $\MF$ by homeomorphisms.
According to the Nielsen--Thurston classification (\cite{Nielsen_classification}, \cite{thurston1988classification}, cf. \cite[Expos{\'e} 9]{1979travaux}), 
a mapping class $g \in \Mod(S)$ is called \emph{pseudo-Anosov} if there exist transverse measured foliations $(\Fc_u, \mu_u), (\Fc_s, \mu_s) \in \MF$ and a real number $\la_g > 1$ such that
$$
g \cdot (\Fc_u, \mu_u) = (\Fc_u, \la_g \cdot \mu_u) \quad \text{and} \quad g \cdot (\Fc_s, \mu_s) = (\Fc_s, \la_g^{-1} \cdot \mu_s).
$$
We call $(\Fc_u, \mu_u)$ and $(\Fc_s, \mu_s)$ the \emph{unstable and stable measured foliations} for $g$ respectively, and they are uniquely determined up to scaling of the transverse measures. The real number $\la_g$ is called the \emph{stretch factor} of $g$, and it satisfies
$$
\len(g) = \log \la_g
$$
(\cite{bers1978extremal}, \cite{thurston1988classification}). Note that $\la_{g^{-1}} = \la_g$.

The space $\PMF = \PMF(S)$ of projective measured foliations on $S$ is defined by
$$
\PMF := \MF / \sim
$$
where the projectivization is given by scaling transverse measures. Thurston compactified the Teichm\"uller space $\Tc$ using $\PMF$ as its boundary. The compactification $\Tc \sqcup \PMF$ is now referred to as the Thurston compactification, and hence $\PMF$ is also referred to as the Thurston boundary. Thurston also showed that the $\Mod(S)$-action on $\Tc$ and the induced $\Mod(S)$-action on $\PMF$ are glued together and give rise to the $\Mod(S)$-action on $\Tc \sqcup \PMF$ by homeomorphisms (\cite{thurston1988classification}, \cite{MR1435975}).

For a pseudo-Anosov $g \in \Mod(S)$, we choose 
$$g^+, g^- \in \MF$$
 unstable and stable measured foliations for $g$ respectively, which are not unique but we make certain choices.
It follows from the above discussion that their projective classes
$$[g^+], [g^-] \in \PMF$$
are distinct fixed points of $g$ acting on $\Tc \sqcup \PMF$. 
Moreover, $g$ exhibits the north-south dynamics: for $x \in \Tc \sqcup \PMF \smallsetminus \{ [g^{\pm}] \}$,
$$
g^n x \to [g^+] \quad \text{as } n \to + \infty \quad \text{and} \quad g^n x \to [g^-] \quad \text{as } n \to - \infty
$$
and the convergence is uniform on compact subsets \cite[Section 4]{mccarthy1985a-tits-alternative}  (cf. \cite[Theorem 3.5]{ivanov1992subgroups}). 
Note also that $[(g^{-1})^+] = [g^-]$ and $[(g^{-1})^-] = [g^+]$.

Two pseudo-Anosov $g, h \in \Mod(S)$ have either the common fixed points (i.e., $\{[g^{\pm}]\} = \{[h^{\pm}]\}$) or disjoint fixed points (i.e., $\{[g^{\pm}]\} \cap \{[h^{\pm}]\} = \emptyset$) \cite[Lemma 2.5]{mccarthy1989dynamics}. We say that they are \emph{independent} when they have disjoint set of fixed points in $\PMF$.
A subgroup $\Ga < \Mod(S)$ is called \emph{non-elementary} if there exist independent pseudo-Anosov $g, h \in \Ga$.

\subsection{Cross-ratios on $\MF$ and $\PMF$}

We prove Theorem \ref{thm:main} by introducing the notion of  cross-ratios on $\MF$ and $\PMF$.

For two isotopy classes $\alpha, \beta$ of simple closed curves on $S$, we denote by $i(\alpha, \beta) \in \Rb_{\ge 0}$ their \emph{geometric intersection number}, i.e., minimal number of intersection points of their representatives. Isotopy classes of (essential) simple closed curves on $S$ can be regarded as measured foliations whose transverse measures are given by intersection numbers with them, by foliating their annular neighborhoods with closed leaves isotopic to those simple closed curves.
In this regard, the function $i (\cdot, \cdot)$ continuously extends to
$$
i : \MF \times \MF \to \Rb_{\ge 0}
$$
which we also call geometric intersection number (\cite[Section 9.3]{thurston0the-geometry}, \cite[Corollary 1.11]{rees1981an-alternative}). Then $i (\cdot, \cdot)$ is invariant under the $\Mod(S)$-action and equivariant under the scaling: for $x, y \in \MF$, $g \in \Mod(S)$, and $t, s > 0$, we have
$$
i(g \cdot x, g\cdot y) = i(x, y) \quad \text{and} \quad i(t\cdot x, s\cdot y) = ts \cdot i(x, y)
$$
where $t \cdot x$ and $s \cdot y$ are obtained by scaling transverse measures of $x$ and $y$ by $t$ and $s$ respectively. 

For a pseudo-Anosov mapping class $g \in \Mod(S)$, we have 
$$i(x, g^-) > 0 \quad \text{for every }x \in \MF \text{ with } [x] \neq [g^-]$$ (\cite[Th{\'e}or{\`e}me 1, Expos{\'e} 12]{1979travaux}, \cite[Lemme 6, Lemme 16, Expos{\'e} 9]{1979travaux}, \cite[Theorem 1.12]{rees1981an-alternative}). The same property also holds for $g^+$.

\begin{definition}[Cross-ratios] \label{def:CR}
    For $x, y, z, w \in \MF$, we define their \emph{cross-ratio}
    $$
[x, y, z, w] =  \frac{i(x, w) i(y, z)}{i(x, z)i(y, w)} \in \Rb_{\ge 0}
$$
whenever $i(x, z) i(y, w) \neq 0$. Note that this is invariant under scaling transverse measures, and hence the cross-ratio is well-defined on $\PMF$ as well.
\end{definition}

In Proposition \ref{prop:CR as limit} below, we relate the cross-ratio to dynamics of pseudo-Anosov mapping classes. We begin with the following lemma, which might be standard to experts; we present the proof for the sake of completeness. With an extra element ``$0$'', we give a topology on $\MF \sqcup \{0\}$ by setting that a sequence $\{x_n\}_{n \in \Nb} \subset \MF$ converges to $0$ if the associated sequence of transverse measures converge to the zero measure.

\begin{lemma} \label{lem:iteratepA}
    Let $g \in \Mod(S)$ be pseudo-Anosov. Then the following holds:
    \begin{enumerate}
        \item For any $x \in \MF$, we have
    $$
    \la_g^{-n} g^n x \to \frac{i(x, g^-)}{i(g^+, g^-)} \cdot g^+ \in \MF \sqcup \{0\} \quad \text{as } n \to + \infty.
    $$
    \item Let $\{x_n\}_{n \in \Nb} \subset \MF$ be a sequence such that, as $n \to + \infty$, we have $x_n \to x \in \MF$ with $i(x, g^-) > 0$. Then 
    $$
    \la_g^{-n} g^n x_n \to \frac{i(x, g^-)}{i(g^+, g^-)} \cdot g^+ \quad \text{as } n \to + \infty.
    $$
    
    \item The convergence in (1) and (2) is uniform on compact subsets of $\MF \smallsetminus~\Rb_{> 0} \cdot g^-$ in the sense that for any compact $Q \subset \MF \smallsetminus \Rb_{> 0} \cdot g^-$ and an open neighborhood $\Uc \subset \MF$ of $\frac{i(Q, g^-)}{i(g^+, g^-)} \cdot g^+ \subset \MF$, 
    $$
    \la_g^{-n} g^n Q \subset \Uc \quad \text{for all large } n \in \Nb.
    $$
    \end{enumerate}

\end{lemma}

\begin{proof}
    We first prove (2). Let $\{x_n\}_{n \in \Nb} \subset \MF$ be a sequence such that $x_n \to x \in \MF$ with $i(x, g^-) > 0$ as $n  \to + \infty$. Since $i(x, g^-) > 0$, we have $[x] \neq [g^-] \in \PMF$. Then by the north-south dynamics of $g$ on $\PMF$, we have
    $$
    [g^n x_n] \to [g^+] \in \PMF \quad \text{as } n \to + \infty.
    $$
    Hence, there exists a sequence $\{t_n\}_{n \in \Nb} \subset \Rb_{> 0}$ such that $t_n g^n x_n \to g^+$ as $n \to + \infty$.

    As $n \to + \infty$, we have
    $$
    i(\la_g^{-n} g^n x_n, g^-) = \la_g^{-n} i(x_n, g^{-n} g^-) = i(x_n, g^-) \to i(x, g^-) > 0.
    $$
  Note that $\la_g^{-n} g^n x_n = (\la_g^{-n} t_n^{-1}) t_n g^n x_n$ for all $n \in \Nb$, and that $t_n g_n x_n \to~g^+$ as $n \to~+ \infty$. It follows from the above computation that
    $$
    \la_g^{-n} t_n^{-1} = \frac{i(\la_g^{-n}g^n x_n, g^-)}{i(t_n g^n x_n, g^-)} \to \frac{i(x, g^-)}{i(g^+, g^-)} \quad \text{as } n \to + \infty.
    $$
    Therefore,
    $$
    \la_g^{-n} g^n x_n = \left(\la_g^{-n} t_n^{-1} \right) t_n g^n x_n \to \frac{i(x, g^-)}{i(g^+, g^-)} \cdot g^+ \quad \text{as } n \to + \infty,
    $$
    as desired.

    \medskip

    We now prove (1). By (2), it suffices to consider the case that $[x] = [g^-] \in \PMF$. Then
    $$
    \la_g^{-n} g^n x = \la_g^{-2n} x \to 0 \quad \text{as } n \to + \infty.
    $$
    Since $i(x, g^-) = 0$ in this case, the desired convergence holds.

    \medskip

    Finally, (3) follows from (2).
\end{proof}

We now deduce the following relation between cross-ratios and stretch factors from Lemma \ref{lem:iteratepA}. Note that in the following, $g^n h^n$ is pseudo-Anosov for all large $n \in \Nb$ since $g$ and $h$ are independent (\cite{mccarthy1985a-tits-alternative}, \cite{ivanov1992subgroups}).

\begin{proposition} \label{prop:CR as limit}
    Let $g, h \in \Mod(S)$ be independent pseudo-Anosov mapping classes. Then 
    $$
    \frac{\la_{g^n h^n}}{\la_{g}^n \la_{h}^n} \to [g^+, h^+, g^-, h^-] \quad \text{as } n \to + \infty.
    $$
\end{proposition}

\begin{proof}
    For each $n \in \Nb$, let $f_n := g^n h^n \in \Mod(S)$, which is pseudo-Anosov when $n$ is large enough. For such $n \in \Nb$, we choose $f_n^{\pm} \in \MF$ such that  
     $i(f_n^+, f_n^-) = 1$.
     It then follows from the north-south dynamics of $g$ and $h$ that
     $$
    [f_n^+] \to [g^+] \in \PMF \quad \text{and} \quad [f_n^-] \to [h^-] \in \PMF \quad \text{as } n \to + \infty.
     $$
    Hence, by rescaling transverse measures of $f_n^{\pm} \in \MF$, we may assume that sequences $\{f_n^{\pm}\}_{n \in \Nb } \subset \MF$ is convergent. We denote their limits by $x, y \in \MF$, i.e.,
    $$
    f_n^+ \to x \quad \text{and} \quad f_n^- \to y \quad \text{as } n \to + \infty.
    $$

    Since $i(f_n^+, f_n^-) = 1$  for all large $n \in \Nb$, we have 
    $$
    \la_{f_n} = i( f_n f_n^+, f_n^-) = i(g^n h^n f_n^+, f_n^-),
    $$
    and hence
    \begin{equation} \label{eqn:CRproof}
    \frac{\la_{g^n h^n}}{\la_{g}^n \la_{h}^n} = i\left( \left(\la_g^{-n} g^n \right) \left(\la_h^{-n} h^n \right) f_n^+, f_n^-\right) = i\left( \la_h^{-n} h^n  f_n^+, \la_g^{-n} g^{-n} f_n^-\right) .
    \end{equation}
    Since $[x] = [g^+]$ and $[y] = [h^-]$, we in particular have
    $$
    i(x, h^-) > 0 \quad \text{and} \quad i(y, g^+) > 0
    $$
    by the independence of $g$ and $h$. Hence, by Lemma \ref{lem:iteratepA}(2), we have as $n \to + \infty$ that 
    $$
    \la_h^{-n} h^n f_n^+ \to \frac{i(x, h^-)}{i(h^+, h^-)} h^+ \quad \text{and} \quad \la_g^{-n} g^{-n} f_n^- \to \frac{i(y, g^+)}{i(g^+, g^-)} g^-.
    $$

    Now since $x = t g^+$ and $y = s h^-$ for some $t, s > 0$, it follows from above convergences and  Equation \eqref{eqn:CRproof} that 
    $$
    \frac{\la_{g^n h^n}}{\la_{g}^n \la_{h}^n} \to \frac{i(x, h^-) i(y, g^+) i(h^+, g^-)}{i(h^+, h^-) i(g^+, g^-)} = ts \cdot \frac{i(g^+,h^-)^2 i(h^+, g^-)}{i(g^+, g^-)i(h^+, h^-)} \quad \text{as } n \to + \infty.
    $$
    Since
    $1 = i(f_n^+, f_n^-) \to i(x, y) = ts \cdot i(g^+, h^-)$ as $n \to +\infty$, we also have 
    $$
    ts = \frac{1}{i(g^+, h^-)}.
    $$
    Therefore,
$$
    \frac{\la_{g^n h^n}}{\la_{g}^n \la_{h}^n} \to \frac{i(g^+,h^-) i(h^+, g^-)}{i(g^+, g^-)i(h^+, h^-)} = [g^+, h^+, g^-, h^-] \quad \text{as } n \to + \infty.
    $$
    This finishes the proof.
\end{proof}

\subsection{Proof of Theorem \ref{thm:main}}

We are now ready to prove Theorem \ref{thm:main}. Let $\Ga < \Mod(S)$ be a non-elementary subgroup and suppose to the contrary that its Teichm\"uller length spectrum $\len(\Ga)$ is not non-arithmetic. That is, for some $c \ge 0$ we have 
$$
\langle \len(\Ga) \rangle \subset c \cdot \Zb.
$$
Then for any independent pseudo-Anosov $g, h \in \Ga$, it follows from Proposition \ref{prop:CR as limit} that
\begin{equation} \label{eqn:arithmetic0}
\log [g^+, h^+, g^-, h^-] \in c \cdot \Zb.
\end{equation}

Denote by
$$\La_{\Ga} := \overline{ \{ [g^+] : g \in \Ga \text{ pseudo-Anosov}\}} \subset \PMF$$ the limit set of $\Ga$. Since $\Ga < \Mod(S)$ is non-elementary, $\La_{\Ga}$ is the unique $\Ga$-minimal subset of $\PMF$ and is perfect \cite[Theorem 4.1, Proposition 5.2]{mccarthy1989dynamics}. Moreover, the north-south dynamics of pseudo-Anosov mapping classes also gives rise to that 
$$
\{ ([g^+], [g^-]) : g \in \Ga \text{ pseudo-Anosov} \} \subset \La_{\Ga} \times \La_{\Ga}
$$
is a dense subset (cf. proof of Proposition \ref{prop:CR as limit}).
 Hence, for $x, y, z, w \in \La_{\Ga}$ such that $i(x, z)i(y, w)i(x, w)i(y, z) > 0$, it follows from Equation \eqref{eqn:arithmetic0} that
\begin{equation} \label{eqn:arithmetic}
\log [x, y, z, w] \in c \cdot \Zb.
\end{equation}
We will deduce a contradiction from Equation \eqref{eqn:arithmetic}.

Let us now turn to the train track theory. We refer the readers to (\cite{thurston0the-geometry}, \cite{PennerHarer_book}, \cite{Papadopoulos_Orsay}) for more comprehensive expository on train tracks and their various properties.
For a maximal and recurrent  train track 
$\tau$ on $S$, let $V_{\tau}$ be the real vector space of weights on the branches of $\tau$ satisfying the switch condition, and let $K(\tau) \subset V_{\tau}$ be the non-empty open convex cone consisting of weights which are positive on every branch of $\tau$. Then there exists a natural  $\Rb_{> 0}$-equivariant map $\phi_{\tau} : K(\tau) \to~\MF$ which is a homeomorphism onto its image. Moreover, the image $U(\tau) := \phi_{\tau}(K(\tau)) \subset \MF$ is a non-empty open cone in $\MF$ whose elements are carried by $\tau$ (cf. \cite[Section~II.3]{Papadopoulos_Orsay}). These $(U(\tau), \phi_{\tau})$'s form an atlas of $\MF$, giving rise to a piecewise-linear structure \cite[Proposition~9.5.8]{thurston0the-geometry} (cf. \cite[Proposition 4.3]{papadopoulos2008measured}). 
Each $(U(\tau), \phi_{\tau})$ is called the train track chart associated to $\tau$.

In the rest of the proof, $$
\text{we fix a pseudo-Anosov } g \in \Ga.
$$ Then, after replacing $g$ with its positive power if necessary, there exists a train track chart $(U(\tau), \phi_{\tau})$ containing $g^- \in \MF$ such that, $g^{-2} U(\tau) \subset U(\tau)$ and the $g^{-2}$-action on $U(\tau)$ induces an action on $V_{\tau}$ given by an invertible linear map $A : V_{\tau} \to V_{\tau}$ with the leading 
eigenvalue $\alpha := \la_g^2$ (\cite[Section~IV]{Papadopoulos_Orsay}, \cite[Theorem 4.1]{papadopoulos1987a-characterization}).

Since $\La_{\Ga}$ is perfect, there exists a pseudo-Anosov $h \in \Ga$ such that $g$ and $h$ are independent and $h^+ \in U(\tau)$; in particular, $i(h^+, g^{\pm}) > 0$. We simply write $z := h^+ \in \MF$. Then for each $n \in \Nb$,
$$
[g^+, g^n z, g^-, g^{-n} z] = \frac{i(g^+, g^{-n} z) i(g^n z, g^-)}{i(g^+, g^-) i(g^n z, g^{-n} z)} = \frac{i(g^+, \la_g^{-n} g^{-n} z) i( \la_g^{-n}  g^n z, g^-)}{i(g^+, g^-) i(\la_g^{-n} g^n z, \la_g^{-n} g^{-n} z)}.
$$
By Lemma \ref{lem:iteratepA}(1), we have $\la_g^{-n} g^{\pm n} z \mapsto \frac{i(z, g^{\mp})}{i(g^+, g^-)} \cdot g^{\pm}$, and hence 
$$
[g^+, g^n z, g^-, g^{-n} z] \to 1 \quad \text{as } n \to + \infty.
$$
On the other hand, since $i(g^+, g^-)i(g^n z, g^{-n} z)i(g^+, g^{-n} z) i(g^n z, g^-) > 0$, it follows from Equation \eqref{eqn:arithmetic} that 
$$
[g^+, g^n z, g^-, g^{-n} z] = 1 \quad \text{for all large } n \in \Nb.
$$
In other words, 
$$
i(g^n z, g^{-n} z) = \frac{i(g^+, g^{-n} z) i(g^n z, g^-)}{i(g^+, g^-)}  \quad \text{for all large } n \in \Nb.
$$
Since $i(g^n z, g^{-n} z) = i(z, g^{-2n} z)$, $i(g^+, g^{-n} z) = i(g^n g^+, z) = \la_g^n i(g^+, z)$, and $i(g^n z, g^-) = i(z, g^{-n} g^-) = \la_g^n i(z, g^-)$, we have 
\begin{equation} \label{eqn:forlarge}
i(z, g^{-2n} z) = \la_g^{2n} \cdot \frac{i(g^+, z) i( z, g^-)}{i(g^+, g^-)}  \quad \text{for all large } n \in \Nb.
\end{equation}

Now noting that $[z] = [h^+] \neq [g^-]$, it follows from (\cite[Proposition 3]{papadopoulos1986geometric}, \cite[Section 3.4]{PennerHarer_book}) that there exists a maximal and recurrent train track $\sigma$ on $S$ and associated train track chart $(U(\sigma), \phi_{\sigma} : K(\sigma) \subset V_{\sigma} \to U(\sigma))$ such that $g^- \in U(\sigma)$ and that the function
$$
i(z, \phi_{\sigma}(\cdot)) \quad \text{is linear on } K(\sigma).
$$
Let $\psi : V_{\sigma} \to \Rb$ be the linear form such that
$$
\psi(\cdot) = i(z, \phi_{\sigma}(\cdot)) \quad \text{on } K(\sigma).
$$
By \cite[Proposition 4.3]{papadopoulos2008measured}, there exists an open cone neighborhood $W \subset K(\tau)$ of $\phi_{\tau}^{-1}(g^-)$ such that $\phi_{\tau}(W) \subset U(\sigma)$ and there exist convex cones $W_1, \dots, W_m \subset W$ with non-empty interiors such that $W = \bigcup_{j = 1}^m W_j$ and 
the transition map
$$
\phi_{\sigma}^{-1} \circ \phi_{\tau} \quad \text{is linear on } W_j \text{ for each } 1 \le j \le m.
$$
For each $j$, let $\phi_j : V_{\tau} \to V_{\sigma}$ be the linear map such that
$$
\phi_j = 
\phi_{\sigma}^{-1} \circ \phi_{\tau} \quad \text{on } W_j.
$$
Define a linear form $L_j : V_{\tau} \to \Rb$ by
$$
L_j := \psi \circ \phi_j.
$$

Now define the function $f_{\tau} : K(\tau) \to \Rb$ by
$$
f_{\tau}(\cdot) = i(z, \phi_{\tau}(\cdot)).
$$
Then we have
$$
f_{\tau} = L_j \quad \text{on } W_j.
$$

We simply write
$$v := \phi_{\tau}^{-1}(z) \in K(\tau).$$
The following observation is crucial in deducing the desired contradiction. See Figure \ref{fig:iceberg} for its name.
\begin{lemma}[Bottom of Iceberg Lemma] 
    \label{lem:nonpositive}
    For each $1 \le j \le m$, we have
$$L_j(v) \le 0.$$
\end{lemma}
\begin{proof}
Let $u \in \inte W_j$ and for $t \in [0, 1]$, let $u_t := (1 - t) u + t v \in K(\tau)$. Then by the convexity of $f_{\tau}$ proven in \cite[Theorem A.1]{mirzakhani} (cf.  \cite[Theorem 4.1]{francisco}), 
$$
f_{\tau}(u_t) \le (1 - t) f_{\tau}(u) + t f_{\tau}(v).
$$
On the other hand, for small enough $t > 0$, $u_t \in W_j$ and hence
$$
f_{\tau}(u_t) = L_j(u_t) = (1- t)L_j(u) + t L_j(v) = (1 - t) f_{\tau}(u) + t L_j(v).
$$
Combining them, we have
$$
L_j(v) \le f_{\tau}(v) = i(z, z) = 0.
$$
\end{proof}

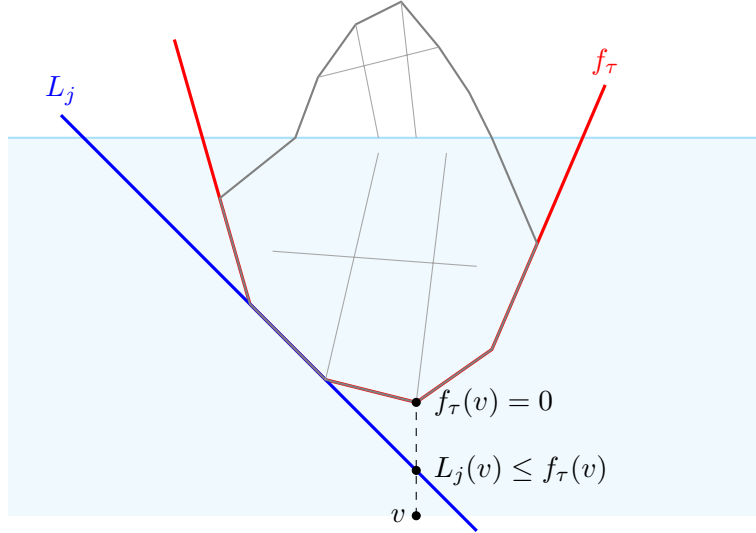
\begin{figure}[h]
    \begin{tikzpicture}

  \fill[white] 
    (-1.2,0) -- (-0.9,0.8) -- (-0.4,1.5) -- (0.2,1.8) 
    -- (0.7,1.2) -- (1.1,0.6) -- (1.4,0) -- cycle;

  \fill[white] 
    (-1.2,0) -- (-2.2,-0.8) -- (-1.8,-2.2) -- (-0.8,-3.2)
    -- (0.4,-3.5) -- (1.4,-2.8) -- (2.0,-1.4) -- (1.4,0) -- cycle;

  \fill[cyan!10, opacity=0.5] (-5,-5) rectangle (5,0);

  \draw[cyan!30, thick] (-5,0) -- (5,0);

    \draw[very thick, red] (1.4 + 0.6 * 2.5, -2.8 + 1.4 * 2.5) -- (1.4,-2.8) -- (0.4,-3.5)
    -- (-0.8,-3.2) -- (-1.8,-2.2) -- (-1.8 - 0.4*2.5, -2.2 + 1.4*2.5);
    \draw[red] (1.4 + 0.6 * 2.5, -2.8 + 1.4 * 2.5) node[above] {$f_{\tau}$};

    \draw[very thick, blue] (-0.8 + 2,-3.2 - 2) -- (-1.8 - 2.5,-2.2 + 2.5);
    \draw[blue] (-1.8 - 2.5,-2.2 + 2.5) node[above] {$L_j$};

  \draw[gray, thick]
    (-1.2,0) -- (-0.9,0.8) -- (-0.4,1.5) -- (0.2,1.8)
    -- (0.7,1.2) -- (1.1,0.6) -- (1.4,0)
    -- (2.0,-1.4) -- (1.4,-2.8) -- (0.4,-3.5)
    -- (-0.8,-3.2) -- (-1.8,-2.2) -- (-2.2,-0.8) -- cycle;

    \filldraw (0.4, -5) circle(1.5pt);
    \draw (0.4, -5) node[left] {$v$};
    \draw[dashed] (0.4, -5) -- (0.4, -3.5);

  \draw[gray!70]
    (-0.9,0.8) -- (0.7,1.2)
    (-0.4,1.5) -- (-0.1,0)
    (0.2,1.8) -- (0.4,0)
    (-1.5,-1.5) -- (1.2,-1.7)
    (-0.8,-3.2) -- (-0.1,-0.2)
    (0.4,-3.5) -- (0.8,-0.2);


    \filldraw (0.4, -3.5) circle(1.5pt);
    \draw (0.5, -3.5) node[right] {$f_{\tau}(v) = 0$};

    \filldraw (0.4, -4.4) circle(1.5pt);
    \draw (0.5, -4.4)  node[right] {$L_j(v) \le f_{\tau}(v)$};

    \end{tikzpicture}
    \caption{Iceberg on $f_{\tau}$} \label{fig:iceberg}
\end{figure}

Now we set 
$$
C := \frac{i(g^+, z) i(z, g^-)}{i(g^+, g^-)} > 0.
$$
Recalling $\alpha = \la_g^2$, it follows from Equation \eqref{eqn:forlarge} that 
$$
f_{\tau}(A^n v) = C \cdot \alpha^n \quad \text{for all large } n \in \Nb.
$$
Since $[g^{-2n} z] \to [g^-]$ in $\PMF$ as $n \to + \infty$, we have $[A^n v] \to [\phi_{\tau}^{-1}(g^-)]$ as $n \to + \infty$, after projecting to some unit sphere in $K(\tau) \subset V_{\tau}$. Hence,
$$A^n v \in W \quad \text{for all large } n \in \Nb.$$

For each $n \ge 0$, set
$$
P_n := \prod_{j = 1}^m \left( L_j(A^n v) - C \cdot \alpha^n \right).
$$
Then for all large $n \in \Nb$, it follows from $A^n v \in W$ that $L_j(A^n v) = f_{\tau}(A^n v) = C \cdot \alpha^n$ for some $1 \le j \le m$. This implies 
\begin{equation} \label{eqn:vanishingPn}
P_n  = 0 \quad \text{for all large } n \in \Nb.
\end{equation}

Now for each $ 1 \le j \le m$, define  the vector space $V_j := V_{\tau} \times \Rb$ and consider the linear form $\widehat L_j : V_j \to \Rb$ defined by
$$
\widehat L_j (u, s) := L_j(u) - C \cdot s.
$$
We also define $A_j : V_j \to V_j$  by
$$
A_j(u, s) := (A u, \alpha s),
$$
which is an invertible linear map.

We then consider tensor products
$$
\widehat V := \bigotimes_{j = 1}^m V_j, \quad \widehat L := \bigotimes_{j = 1}^m \widehat L_j, \quad \text{and} \quad \widehat A := \bigotimes_{j = 1}^m A_j.
$$
Then $\widehat L : \widehat V \to \Rb$ is a linear form and $\widehat A : \widehat V \to \widehat V$ is an invertible linear map.

Let  
$$
\chi_{\widehat A}(t) := t^{d} + c_{d-1} t^{d-1} + \cdots + c_1 t + c_0
$$
be the characteristic polynomial of $\widehat A$. Since $\widehat A$ is invertible, $c_0 \neq 0$.
By Cayley--Hamilton theorem, 
$$
\widehat A^d + c_{d-1} \widehat A^{d-1} + \cdots + c_1 \widehat A + c_0 \operatorname{Id} = 0.
$$
In other words, for each $n \ge 0$,
$$
\widehat A^{n+d} + c_{d-1} \widehat A^{n + d-1} + \cdots + c_1 \widehat A^{n+1} + c_0 \widehat A^n = 0.
$$
In particular, setting
$$
\widehat v := \bigotimes_{j = 1}^m (v, 1) \in \widehat V,
$$
we have 
$$
\widehat A^{n+d} \widehat v + c_{d-1} \widehat A^{n + d-1} \widehat v + \cdots + c_1 \widehat A^{n+1} \widehat v + c_0 \widehat A^n \widehat v = 0,
$$
and hence
$$
\widehat L \left(\widehat A^{n+d} \widehat v\right) + c_{d-1} \widehat L \left(\widehat A^{n+d - 1} \widehat v\right) + \cdots + c_1 \widehat L \left(\widehat A^{n+1} \widehat v\right) + c_0 \widehat L \left(\widehat A^{n} \widehat v\right) = 0.
$$

Note that for each $n \ge 0$,
$$
\widehat A^n \widehat v = \bigotimes_{j = 1}^m (A^n v, \alpha^n)
$$
and hence
$$
\widehat L \left( \widehat A^n \widehat v \right) = \prod_{j = 1}^m \left( L_j (A^n v) - C \cdot \alpha^n \right) = P_n.
$$
Therefore, it follows that
$$
P_{n+d} + c_{d-1} P_{n + d-1} + \cdots + c_1 P_{n+1} + c_0 P_n = 0 \quad \text{for each } n \ge 0.
$$

On the other hand, as in Equation \eqref{eqn:vanishingPn}, $P_n = 0$ for all large $n \ge 0$. Therefore, since $c_0 \neq 0$, backward induction gives $P_n = 0$ for all $n \ge 0$. In particular,
$$
P_0 = 0.
$$
Since $P_0 = \prod_{j = 1}^m \left(L_j(v) - C \right)$, this implies
$$
L_j(v) - C  = 0 \quad \text{for some } 1 \le j \le m.
$$
However, by Bottom of Iceberg Lemma (Lemma \ref{lem:nonpositive}), $L_j(v) \le 0$. This is a contradiction to $C > 0$, completing the proof of Theorem \ref{thm:main}.
\qed

\bigskip

\bibliographystyle{alpha} 
\bibliography{ML}

\end{document}